\documentclass[a4paper,reqno]{amsart}

\newtheorem{theorem}{Theorem}
\newtheorem{lemma}[theorem]{Lemma}
\newtheorem{remark}[theorem]{Remark}
\newtheorem{example}[theorem]{Example}

\begin{document}


\title{symmetric affine surfaces with torsion}

\author{D.~D'Ascanio, P.~B.~Gilkey and P.~Pisani}
\address{D'Ascanio: Instituto de F\'isica La Plata, CONICET and Universidad Nacional de La Plata, CC 67 (1900) La Plata, Argentina. }
\email{dascanio@fisica.unlp.edu.ar}
\address{Gilkey: Mathematics Department, University of Oregon, Eugene OR 97403 USA. }
\email{gilkey@uoregon.edu}
\address{Pisani: Instituto de F\'isica La Plata, CONICET and Universidad Nacional de La Plata, CC 67 (1900) La Plata, Argentina. }
\email{pisani@fisica.unlp.edu.ar}

\begin{abstract}
We study symmetric affine surfaces which have non-vanishing torsion tensor.
We give a complete classification of the local geometries possible if the torsion is assumed
parallel. This generalizes a previous result of Opozda in the torsion free setting; these geometries
are all locally homogeneous. If the torsion is not parallel, we assume the underlying surface is locally homogeneous and provide
a complete classification in this setting as well.
\end{abstract}

\maketitle


\thispagestyle{empty}

\section{Introduction and statement of results}

In differential geometry, a connection $\nabla$ on the tangent bundle of a smooth manifold $M$ gives rise to the notion of parallelism.
The pair $\mathcal{M}=(M,\nabla)$ is called an {\it affine manifold}; if $\dim\{M\}=2$, $\mathcal{M}$ is called
an {\it affine surface}. We emphasize that in contrast to the usage
employed by some authors, we permit the torsion tensor $T(X,Y):=\frac12(\nabla_XY-\nabla_YX-[X,Y])$ to be non-zero.

The study of various properties of affine manifolds is relevant in {\it non-metric extensions} of General Relativity, i.e.~geometries where the connection $\nabla$
does not arise as the Levi-Civita connection of some underlying pseudo-Riemannian metric.
The standard formulation of General Relativity regards the metric as a
canonical field which determines the affine structure by means of the Levi-Civita connection. Open questions
in our current understanding of gravitation have led physicists to study generalizations of this scenario.
In non-metric extensions of General Relativity \cite{Hehl:1994ue,Zanelli:2005sa}, the affine connection provides
an independent degree of freedom; in particular, Einstein-Cartan theory regards the torsion tensor as a new canonical field.

Spacetimes with torsion give different dynamics for matter fields \cite{Shapiro:2001rz} (see \cite{Hammond:2002rm}
for an account of experiments aimed at measuring the existence of torsion). As a dynamical field, torsion also plays an
important role in alternative models of the early universe \cite{Poplawski16,Trautman:1973wy}. For recent articles on
Einstein-Cartan gravity, we refer to \cite{BeltranJimenez:2019tjy,Klemm:2018bil}. Note also that two-dimensional theories
of gravity constitute an area of interest on its own; for studies of torsion in this context see
\cite{Grumiller:2002nm,Katanaev:1986qm,Obukhov:1997uc} and the references therein.
Finally, non-metric connections  can be used to study defects in condensed matter; in this setting, the
torsion describes dislocations in solids
\cite{Hehl:2007bn,Katanaev:1992kh,Kroener90}.
Thus, apart from their purely mathematical relevance, the affine properties of manifolds are of interest in physical contexts.

\subsection{Notational conventions} A diffeomorphism of the underlying manifold $M$ is said to be
an {\it affine diffeomorphism} if it preserves the connection; the geometry
$\mathcal{M}$ is said to be {\it affine homogeneous} if the group of affine diffeomorphisms acts transitively. There is a corresponding
local theory.
Let $R(X,Y):=\nabla_X\nabla_Y-\nabla_Y\nabla_X-\nabla_{[X,Y]}$ be the {\it curvature operator} of an affine geometry
$\mathcal{M}=(M,\nabla)$. We contract indices to define the {\it Ricci tensor}
$\rho(X,Y):=\operatorname{Trace}(Z\rightarrow R(Z,X)Y)$.
We say that $\mathcal{M}$ has a {\it symmetric Ricci tensor}
if $\rho(X,Y)=\rho(Y,X)$ for all $X$ and $Y$; this is always the case in the metrizable setting but need not hold in general.
We say that $\mathcal{M}$ is a {\it symmetric affine surface} if $\mathcal{M}$ is an affine surface
satisfying $\nabla R=0$ or,  equivalently as we are in the 2-dimensional setting, if the Ricci tensor is parallel, i.e. $\nabla\rho=0$.
We will show presently in Lemma~\ref{L7} that any symmetric affine surface has a symmetric Ricci tensor.

If $(x^1,x^2)$ is a system of local coordinates on an
affine surface, we expand $\nabla_{\partial_{x^i}}\partial_{x^j}=\Gamma_{ij}{}^k\partial_{x^k}$; the {\it Christoffel symbols}
$\Gamma_{ij}{}^k$ determine the connection and we shall specify geometries by giving their (possibly) non-zero Christoffel symbols.
We say that an affine surface $\mathcal{M}_1=(M_1,\nabla_1)$ is {\it modeled} on an affine surface $\mathcal{M}=(M,\nabla)$ if $\mathcal{M}$ is
homogeneous and if there is a cover of $M_1$ by open sets which are affine isomorphic
to open subsets of $M$. This implies that $\mathcal{M}_1$ is locally homogeneous.

\subsection{Symmetric affine surfaces with vanishing torsion}
We say $\mathcal{M}$ is {\it torsion free}
if $T=0$. The torsion free symmetric affine surfaces have been classified by Opozda~\cite{Op91}.
The Ricci tensor is symmetric and there are 6 possible signatures.
If $\operatorname{Rank}\{\rho\}=0$, then $\rho=0$; if $\operatorname{Rank}\{\rho\}=1$, then $\rho$ is either positive semi-definite
($\rho\ge0$)
or negative semi-definite ($\rho\le0$);
if $\operatorname{Rank}\{\rho\}=2$, then $\rho$ is either positive definite ($\rho>0$), negative definite ($\rho<0$), or indefinite.
The symmetric affine surfaces without torsion are all locally homogeneous and
modeled on one of six non-isomorphic geometries which are distinguished by the signature of the Ricci tensor.
The first four of the geometries, given in Assertions~(1--4) below, are metrizable, i.e. the connection is the associated
Levi-Civita connection.
The remaining two geometries, given in Assertions~(5,6) below, are not metrizable.

\begin{theorem}[Opozda]\label{T1} Let $\mathcal{M}$ be a symmetric affine surface without torsion.
Then $\mathcal{M}$ is locally homogeneous and modeled on one of the following geometries:
\begin{enumerate}
\item The flat plane $\mathbb{R}^2$ with $ds^2=(dx^1)^2+(dx^2)^2$; $\rho=0$.
\item The hyperbolic plane $\mathbb{R}^+\times\mathbb{R}$ with $ds^2=\frac{(dx^1)^2+(dx^2)^2}{(x^1)^2}$; $\rho<0$.
\item The Lorentzian plane $\mathbb{R}^+\times\mathbb{R}$ with $ds^2=\frac{(dx^1)^2-(dx^2)^2}{(x^1)^2}$;
$\rho$ is indefinite.
\item The round sphere $S^2$; $\rho>0$.
\item The non-metrizable geometry  with $\Gamma_{11}{}^1=1$ and $\Gamma_{22}{}^1=+1$; $\rho\ge0$.
\item The non-metrizable geometry with $\Gamma_{11}{}^1=1$ and $\Gamma_{22}{}^1=-1$;
$\rho\le0$.\end{enumerate}\end{theorem}

\subsection{Symmetric affine surfaces with non-vanishing and parallel torsion}
We will prove the following result in Section~\ref{S2} extending Theorem~\ref{T1}.
Were we to take $u=0$ in Theorem~\ref{T2}~(2), then the torsion would vanish and we would obtain the geometries described in
Theorem~\ref{T1}~(5,\ 6).

\begin{theorem}\label{T2}
Let $\mathcal{M}$ be a symmetric affine surface with non-vanishing and parallel torsion.  Then $\mathcal{M}$
is modeled on one of the following structures.\begin{enumerate}
\item $\Gamma_{11}{}^1=1$, $\Gamma_{12}{}^1=2$, $\rho=0$.
\item $\Gamma_{11}{}^1=1$, $\Gamma_{12}{}^1=2u$ for $u>0$,
$\Gamma_{22}{}^1=\pm1$, $\rho=\pm\, dx^2\otimes dx^2$.
\end{enumerate}
These geometries are all inequivalent affine structures and homogeneous.
\end{theorem}

\subsection{Locally homogeneous affine surfaces}

Opozda~\cite{Op04} classified the locally homogeneous affine surfaces without torsion.
Subsequently,  Arias-Marco and Kowalski~\cite{AMK08}
extended this classification to the more general setting;
a different proof of this result has been given recently by Brozos-V\'azquez et al.~\cite{BGG19}. Previous studies of locally homogeneous surfaces in the torsion free setting include \cite{OK-BO-ZV,OK-ZV}. For a different approach in higher dimensions we refer to \cite{ZD-OK}.

\begin{theorem}\label{T3}
Let $\mathcal{M}$ be a locally homogeneous affine surface, possibly with torsion.
At least one of the following possibilities holds:
\begin{enumerate}
\item Type~$\mathcal{A}$: There is a coordinate atlas for
$\mathcal{M}$ so $\Gamma_{ij}{}^k\in\mathbb{R}$.
\item Type~$\mathcal{B}$: There is a coordinate atlas
for $\mathcal{M}$ so $x^1\Gamma_{ij}{}^k\in\mathbb{R}$ and $x^1>0$.
\item Type~$\mathcal{C}$:
The geometry is locally isomorphic to the geometry of the round sphere $S^2$ with the associated
Levi-Civita connection.
\end{enumerate}\end{theorem}

The possibilities of Theorem~\ref{T3} are not exclusive; there are geometries which can be realized both as Type~$\mathcal{A}$
and Type~$\mathcal{B}$ structures. However, no Type~$\mathcal{A}$ or Type~$\mathcal{B}$ structure is also Type~$\mathcal{C}$.
We refer to Calvi\~no-Louzao et al.~\cite{CGGPR19} for additional information in this regard.

We now examine affine symmetric surfaces with non-parallel torsion; to obtain a useful classification, we shall not consider
the most general surfaces but restrict to locally homogeneous geometries. Theorem~\ref{T4} (resp. Theorem~\ref{T5}), to be
proved in Section~\ref{S3} (resp. Section~\ref{S4}) deals with surfaces of Type~$\mathcal{A}$ (resp. Type~$\mathcal{B}$).

\subsection{Type~$\mathcal{A}$ affine symmetric surfaces}
The general linear group $\operatorname{GL}(2,\mathbb{R})$ acts on the set of Type~$\mathcal{A}$
geometries by change of basis. We will establish the following result in Section~\ref{S3} which classifies
the Type~$\mathcal{A}$ symmetric affine surfaces with non-parallel torsion.

\begin{theorem}\label{T4}
 Let $\mathcal{M}$ be a Type~$\mathcal{A}$ symmetric affine surface with non-parallel torsion tensor. Then $\mathcal{M}$
 is flat (i.e. $\rho=0$) and $\mathcal{M}$ is equivalent under the action of the gauge group
 $\operatorname{GL}(2,\mathbb{R})$ to one of the following 6 geometries for
 $\alpha,\eta\in\mathbb{R}$, $\beta\in\mathbb{R}-\{0,2\}$,  $\gamma\in\mathbb{R}-\{0\}$,
 $\varepsilon=\pm1$, $\eta\geq0$,
and $T=(dx^1\wedge dx^2)\otimes\partial_{x^2}$, where no two different surfaces are linearly equivalent:
 \begin{enumerate}
  \item$\begin{array}[t]{lllll}\Gamma_{11}{}^1=\gamma,&\Gamma_{11}{}^2=\gamma-1,&\Gamma_{12}{}^1=0,&\Gamma_{12}{}^2=1,\\\Gamma_{21}{}^1=0,
  &\Gamma_{21}{}^2=-1,&\Gamma_{22}{}^1=0,&\Gamma_{22}{}^2=1,\end{array}$
  \quad$\nabla T=\left(\begin{array}{cc} 0 & -\gamma  \\ 0 & 0 \\\end{array}\right)$.\\
  \smallbreak
  \item$\begin{array}[t]{lllll}\Gamma_{11}{}^1=0,&\Gamma_{11}{}^2=\alpha,\hphantom{a....}
  &\Gamma_{12}{}^1=1,&\Gamma_{12}{}^2=2,\\
  \Gamma_{21}{}^1=1,&\Gamma_{21}{}^2=0,&\Gamma_{22}{}^1=0,&\Gamma_{22}{}^2=1,\end{array}$
  \quad$\nabla T=\left(\begin{array}{cc} 1 & 0  \\ 0 & -1 \\\end{array}\right)$.
  \smallbreak
  \item$\begin{array}[t]{lllll}\Gamma_{11}{}^1=\gamma,&\Gamma_{11}{}^2=0,\hphantom{.}&
  \Gamma_{12}{}^1=0,&\Gamma_{12}{}^2=\gamma,\\
  \Gamma_{21}{}^1=0,&\Gamma_{21}{}^2=\gamma -2,&\Gamma_{22}{}^1=0,&\Gamma_{22}{}^2=0,\end{array}$
  \quad\hphantom{.}$\nabla T=\left(\begin{array}{cc} 0 & -\gamma \\ 0 & 0 \\\end{array}\right)$.
  \smallbreak
  \item$\begin{array}[t]{lllll}\Gamma_{11}{}^1=2,&\Gamma_{11}{}^2=1,\hphantom{......}
  &\Gamma_{12}{}^1=0,&\Gamma_{12}{}^2=2,\\
  \Gamma_{21}{}^1=0,&\Gamma_{21}{}^2=0,&\Gamma_{22}{}^1=0,&\Gamma_{22}{}^2=0,\end{array}$
  \quad\hphantom{.}$\nabla T=\left(\begin{array}{cc} 0 & -2 \\ 0 & 0 \\\end{array}\right)$.
  \smallbreak
  \item$\begin{array}[t]{lllll}\Gamma_{11}{}^1=\beta,&\Gamma_{11}{}^2=0,\phantom{......}
  &\Gamma_{12}{}^1=0,&\Gamma_{12}{}^2=2,\\
  \Gamma_{21}{}^1=0,&\Gamma_{21}{}^2=0,&\Gamma_{22}{}^1=0,&\Gamma_{22}{}^2=0,\end{array}$
  \quad\hphantom{.}$\nabla T=\left(\begin{array}{cc} 0 & -\beta  \\ 0 & 0 \\\end{array}\right)$.
  \smallbreak
  \item$\begin{array}[t]{lllll}\Gamma_{11}{}^1=\omega,&\Gamma_{11}{}^2=0,&\Gamma_{12}{}^1=0,&\Gamma_{12}{}^2=\omega,\\
  \Gamma_{21}{}^1=0,&\Gamma_{21}{}^2=\omega -2,&\Gamma_{22}{}^1=\varepsilon,&\Gamma_{22}{}^2=\eta,\end{array}$
  \quad$\nabla T=\left(\begin{array}{cc} 0  & -\omega  \\ \varepsilon  & 0  \\\end{array}\right)$.
 \end{enumerate}
\end{theorem}

\begin{remark}\rm
	We choose $\eta\geq 0$ because any surface of family (6) given by $(\omega,\eta)$ is equivalent to $(\omega,-\eta)$ thru $x_2\to-x_2$. The constraint $\beta\neq 2$ in the surfaces of family (5) ensures non-equivalence with the surface (3) with $\gamma=2$.
\end{remark}

\subsection{Type~$\mathcal{B}$ affine symmetric surfaces}
If $\tilde\Gamma_{ij}{}^k\in\mathbb{R}$, we construct a Type~$\mathcal{B}$ geometry by setting
$\Gamma_{ij}{}^k=\frac1{x^1}\tilde \Gamma_{ij}{}^k$.
To simplify denominators, we evaluate at $x^1=1$ to define $\tilde\rho$, $\tilde T$, and $\widetilde{\nabla T}$. We then have $\rho=(x^1)^{-2}\tilde\rho$, $T=(x^1)^{-1}\tilde T$, and
$\nabla T=(x^1)^{-2}\widetilde{\nabla T}$.
The $ax+b$ group acts on the set of Type~$\mathcal{B}$ geometries
by the linear change of basis $(x^1,x^2)\rightarrow(x^1,ax^2+bx^1)$. In Section~\ref{S4}, we complete our classification of the
locally homogeneous symmetric affine surfaces by establishing the following result.

\begin{theorem}\label {T5}
Let $\mathcal{M}$ be a Type~$\mathcal{B}$ symmetric affine surface with non-parallel torsion
tensor. Then $\mathcal{M}$ is equivalent under the action of the
$ax+b$ gauge group to one of the following 9 structures with associated parameters
 $\xi,\ \eta,\alpha,\ \beta,\ \gamma,\ \delta\in\mathbb{R}$ for
  $\alpha\geq 0$; $\eta+1\neq2\delta$; $\gamma\neq-\frac12$. The torsion is given by
$T=(x^1)^{-1}(dx^1\wedge dx^2)\otimes(T^1\partial_{x^1}+T^2\partial_{x^2})$. No two different surfaces in this classification are linearly equivalent.
\begin{enumerate}
\setlength\itemsep{2mm}
\addtolength{\itemindent}{-.7cm}
\item$\begin{array}[t]{llllll}
\tilde\Gamma_{11}{}^1=-2, &\tilde\Gamma_{11}{}^2=\xi , &\tilde\Gamma_{12}{}^1=0, &\tilde\Gamma_{12}{}^2=0 , & \tilde\Gamma_{21}{}^1=-1,\\
\tilde\Gamma_{21}{}^2= \xi, &\tilde\Gamma_{22}{}^1=0, &\tilde\Gamma_{22}{}^2=1,
& T^1=\frac12, & T^2= -\frac{\xi}{2},
\end{array}$
\medbreak
\noindent$\tilde\rho=\left(
\begin{array}{cc}
 \xi  & 1 \\
1 & 0 \\
\end{array}
\right)$,
\quad
$\widetilde{\nabla T}=-\frac12\left(
\begin{array}{cc}
1 & 0  \\
1 & 0
\end{array}
\right)$.
\smallbreak
\item$\begin{array}[t]{lllll}\tilde\Gamma_{11}{}^1=\eta,&
 \tilde\Gamma_{11}{}^2=0,&\tilde\Gamma_{12}{}^1=0,&
  \tilde\Gamma_{12}{}^2=\eta + 1 ,&\tilde\Gamma_{21}{}^1=0,\\
  \tilde\Gamma_{21}{}^2=\eta +1 -2\delta,& \tilde\Gamma_{22}{}^1=0,&
  \tilde\Gamma_{22}{}^2=0,&T^1=0,& T^2=\delta, \end{array}$
  \medbreak
  \noindent$\rho=0$,
  \quad$\widetilde{\nabla T}=\left(
  \begin{array}{cc}
  0 & { -\delta  (1 +\eta)} \\
  0 & 0 \\
  \end{array}
  \right)$.
  \smallbreak
\item$\begin{array}[t]{lllll}\tilde\Gamma_{11}{}^1=2\beta -2,&
  \tilde\Gamma_{11}{}^2=1,&\tilde\Gamma_{12}{}^1=0,&
  \tilde\Gamma_{12}{}^2=2\beta-1 ,&\tilde\Gamma_{21}{}^1=0,\\
  \tilde\Gamma_{21}{}^2=-1,& \tilde\Gamma_{22}{}^1=0,&
  \tilde\Gamma_{22}{}^2=0,&T^1=0,& T^2=\beta, \end{array}$
  \medbreak
  \noindent$\rho=0$,
  \quad$\widetilde{\nabla T}=\left(
  \begin{array}{cc}
  0 & -\beta  (2\beta -1) \\
  0 & 0 \\
  \end{array}
  \right)$.
  \smallbreak
\item$\begin{array}[t]{lllll}
  \tilde\Gamma_{11}{}^1=0 ,&\tilde\Gamma_{11}{}^2=\xi ,
  &\tilde\Gamma_{12}{}^1=1,& \tilde\Gamma_{12}{}^2=2 \beta ,&\tilde\Gamma_{21}{}^1=0,\\
  \tilde\Gamma_{21}{}^2=0,&\tilde\Gamma_{22}{}^1=0,&\tilde\Gamma_{22}{}^2=0,&
  T^1=\frac12 ,&T^2=\beta, \end{array}$
  \medbreak
  \noindent$\rho=0$,
  \quad
  $\widetilde{\nabla T}=\frac12\left(
  \begin{array}{cc}
  -1 & { \xi-2\beta }  \\
  0 & 0 \\
  \end{array}
  \right)$.
  \smallbreak
\item$\begin{array}[t]{lllll}
  \tilde\Gamma_{11}{}^1=\xi ,&\tilde\Gamma_{11}{}^2=0 ,
  &\tilde\Gamma_{12}{}^1=0,& \tilde\Gamma_{12}{}^2=2 \beta ,&\tilde\Gamma_{21}{}^1=0,\\
  \tilde\Gamma_{21}{}^2=0,&\tilde\Gamma_{22}{}^1=0,&\tilde\Gamma_{22}{}^2=0,&
  T^1=0 ,&T^2=\beta, \end{array}$
  \medbreak
  \noindent$\rho=0$,
  \quad$\widetilde{\nabla T}=\left(
  \begin{array}{cc}
  0 & { -\beta(1+\xi) }  \\
  0 & 0 \\
  \end{array}
  \right)$.
  \smallbreak
\item$\begin{array}[t]{lllll}
  \tilde\Gamma_{11}{}^1=2\beta ,&\tilde\Gamma_{11}{}^2=1 ,
  &\tilde\Gamma_{12}{}^1=0,& \tilde\Gamma_{12}{}^2=2 \beta ,&\tilde\Gamma_{21}{}^1=0,\\
  \tilde\Gamma_{21}{}^2=0,&\tilde\Gamma_{22}{}^1=0,&\tilde\Gamma_{22}{}^2=0,&
  T^1=0 ,&T^2=\beta, \end{array}$
  \medbreak
  \noindent$\rho=0$,
  \quad
  $\widetilde{\nabla T}=\left(
  \begin{array}{cc}
  0 & { -\beta(1+2\beta) }  \\
  0 & 0 \\
  \end{array}
  \right)$.
  \smallbreak
\item$\begin{array}[t]{lllll}\tilde\Gamma_{11}{}^1=\xi ,&\tilde\Gamma_{11}{}^2=0,&
  \tilde\Gamma_{12}{}^1=2 \alpha,& \tilde\Gamma_{12}{}^2=-1,&\tilde\Gamma_{21}{}^1=0,\\
  \tilde\Gamma_{21}{}^2=0,& \tilde\Gamma_{22}{}^1=\varepsilon ,&\tilde\Gamma_{22}{}^2=0,&
  T^1=\alpha ,& T^2=-\frac{1}{2},\end{array}$
  \medbreak
  \noindent$\tilde\rho=\left(
  \begin{array}{cc}
  0 & 0 \\
  0 & \varepsilon\xi    \\
  \end{array}
  \right)$,
  \quad$\widetilde{\nabla T}=-\frac12\left(
  \begin{array}{cc}
  2\alpha &  -\xi-1\\
  \varepsilon & 0 \\
  \end{array}
  \right)$.
  \smallbreak
\item$\begin{array}[t]{lllll}\tilde\Gamma_{11}{}^1=2 \gamma +1,&
  \tilde\Gamma_{11}{}^2=0,&\tilde\Gamma_{12}{}^1=2 \alpha,&
  \tilde\Gamma_{12}{}^2=2 \gamma ,&\tilde\Gamma_{21}{}^1=0,\\
  \tilde\Gamma_{21}{}^2=0,& \tilde\Gamma_{22}{}^1=\varepsilon ,&
  \tilde\Gamma_{22}{}^2=0,&T^1=\alpha ,&T^2=\gamma, \end{array}$
  \medbreak
  \noindent$\rho=0$,
  \quad$\widetilde{\nabla T}=\left(
  \begin{array}{cc}
  -\alpha &  -2\gamma  (\gamma +1) \\
  \varepsilon\gamma  & 0 \\
  \end{array}
  \right)$.
  \smallbreak
\item$\begin{array}[t]{lllll}\tilde\Gamma_{11}{}^1=-1,&\tilde\Gamma_{11}{}^2
  =-{2\varepsilon \alpha (2   \gamma +1 )},&\tilde\Gamma_{12}{}^1=2 \alpha,&
  \tilde\Gamma_{12}{}^2=-1,\\\tilde\Gamma_{21}{}^1=0,&\tilde\Gamma_{21}{}^2=-2 \gamma -1,&
  \tilde\Gamma_{22}{}^1=\varepsilon ,&\tilde\Gamma_{22}{}^2=0,\\T^1=\alpha ,& T^2=\gamma, \end{array}$
  \medbreak
  \noindent$\tilde\rho=\left(
  \begin{array}{cc}
  -2 \gamma -1 & 0 \\
  0 & -\varepsilon  \\
  \end{array}
  \right)$,
  \quad$\widetilde{\nabla T}=\left(
  \begin{array}{cc}
  2 \alpha  \gamma  &  -2\varepsilon \alpha ^2 (2 \gamma +1) \\
  \varepsilon\gamma  & -\alpha  (2 \gamma +1) \\
  \end{array}
  \right)$.
 \end{enumerate}
\end{theorem}

\begin{remark}\rm
We choose $\alpha\geq 0$ because any surface of family (7) given by $(\xi,\alpha)$ is equivalent to $(\xi,-\alpha)$ thru $x_2\to-x_2$. 
The same observation holds for families (8) and (9). The remaining constraints on $\eta,\gamma,\delta$ ensure non-equivalence between different families. 
Note that $\gamma=-1/2$ in (8) gives (7) for $\xi=0$, and that $\gamma=-1/2$ in (9) gives (7) for $\xi=-1$.
\end{remark}

\subsection{Symmetric affine surfaces which are not locally homogeneous}
As for the most part, we shall be concerned with locally homogeneous geometries, we conclude the introduction by presenting
two examples of symmetric affine surfaces which are not locally homogeneous.
\goodbreak\begin{example}\rm
\ \begin{enumerate}
\item Let the non-zero symbols be $\Gamma_{12}{}^2=\frac12\tanh(x^1)$ and $\Gamma_{21}{}^2=-\frac12\tanh(x^1)$. Then $\rho=dx^1\otimes dx^1$,
$\nabla\rho=0$, and $\nabla T=(\cosh{x_1})^{-2}\,dx^1\wedge dx^2\otimes\partial_{x^2}$.
The space of affine Killing vector fields is given by
$\operatorname{Span}\{1,x^1,x^2\}\partial_{x^2}$. Thus this is a symmetric affine surface of cohomogeneity 1.
\item Let $\{X,Y\}$ be a frame for the tangent bundle of $M$. There is a unique
 connection with torsion so $\nabla X=0$ and $\nabla Y=0$. Let $\{X^*,Y^*\}$ be the corresponding
 dual frame for the cotangent bundle and let $[X,Y]$ denote the Lie bracket of $X$ and $Y$.
 We have
 $T=\frac12(X^*\wedge Y^*)\otimes[X,Y]$ and the geometry is flat.
\end{enumerate}\end{example}

\section{Proof of Theorem~\ref{T2}}
\label{S2}

\subsection{The Ricci tensor of a symmetric affine surface}
The fact that the Ricci tensor of a symmetric affine surface is a symmetric 2-tensor is due to
Opozda~\cite{Op91} in the torsion free setting; it is not known if a similar statement holds in higher dimensions. We can extend this
result to the setting of affine surfaces with torsion.

\begin{lemma}\label{L7}
If $\mathcal{M}$ is a connected symmetric affine surface, then the Ricci tensor of $\mathcal{M}$ is a symmetric 2-tensor
which has constant rank.
\end{lemma}

\begin{proof}
Extend the action of the curvature operator to tensors of all types.
The {\it alternating Ricci tensor} is defined by setting $\rho_a:=(\rho_{12}-\rho_{21})dx^1\wedge dx^2$.
As the commutator of covariant differentiation is given by curvature, one has:
$$
\begin{array}{l}
\rho_{a;21}-\rho_{a;12}=(\rho_{12}-\rho_{21})R_{12}(dx^1\wedge dx^2)
\\[0.03in]
\phantom{\rho_{a;21}-\rho_{a;12}}
=(\rho_{12}-\rho_{21})(-R_{121}{}^1-R_{122}{}^2)(dx^1\wedge dx^2)
\\[0.03in]
\phantom{\rho_{a;21}-\rho_{a;12}}
=(\rho_{12}-\rho_{21})(-\rho_{21}+\rho_{12})(dx^1\wedge dx^2)\,.
\end{array}
$$
If $\nabla\rho=0$, then $\nabla\rho_a=0$ and thus $(\rho_{12}-\rho_{21})^2=0$ and $\rho_a=0$. This shows that the Ricci tensor of $\mathcal{M}$ is symmetric.

We have assumed that $M$ is connected. Given points $P$ and
$Q$, let $\sigma(t)$ be a curve from $P$ to $Q$. Let $\{e_1(t),e_2(t)\}$ be a parallel frame for the tangent bundle along $\sigma(t)$.
Since $\nabla\rho=0$, we compute
\begin{eqnarray*}
 &&\partial_t\{\rho(e_i(t),e_j(t))\}\\
 &=&\{\nabla_{\partial_t}\rho\}(e_i(t),e_j(t))-\rho(\nabla_{\partial_t}e_i(t),e_j(t))-\rho(e_i(t),\nabla_{\partial_t}e_j(t))=0\,.
\end{eqnarray*}
Thus the matrix of $\rho$ is constant and $\operatorname{Rank}(\rho)$ is constant.
\end{proof}

\subsection{Abstract torsion tensors} Let $(M,\nabla)$ be an affine surface.  Let $(x^1,x^2)$ be a system of local coordinates on $M$.
In terms of the Christoffel symbols the torsion tensor takes the form
$$
T:=\textstyle\frac12(dx^1\wedge dx^2)\otimes\{(\Gamma_{12}{}^1-\Gamma_{21}{}^1)\partial_{x^1}
+(\Gamma_{12}{}^2-\Gamma_{21}{}^2)\partial_{x^2}\}\,.
$$
Let $\mathfrak{T}(M)$ be the vector space of 2-form valued tangent vector fields; $S\in\mathfrak{T}(M)$
if there are smooth functions $S^1$ and $S^2$ so that
$S=(dx^1\wedge dx^2)\otimes(S^1\partial_{x^1}+S^2\partial_{x^2})$.
This is the space of {\it abstract torsion tensors}.
Let $\mathfrak{P}(\mathcal{M}):=\{S\in\mathfrak{T}(\mathcal{M}):\nabla S=0\}$
be the subspace of {\it parallel abstract torsion tensors}.
\begin{lemma}\label{L8}
If $\mathcal{M}$ is an affine surface, then $\operatorname{Rank}\{\rho\}+\dim\{\mathfrak{P}\}\le2$.
\end{lemma}

\begin{proof}
Since $M$ is connected, a parallel tensor is determined by its value at any point of $M$.
Thus $\dim\{\mathfrak{P}\}\le2$. Suppose that
$0\ne S\in\mathfrak{P}(\mathcal{M})$ is a parallel abstract torsion tensor. We compute:
\begin{eqnarray*}
 0&=&S_{;12}-S_{;21}=R_{12}(dx^1\wedge dx^2)\otimes S^k\partial_{x^k}+(dx^1\wedge dx^2)\otimes R_{12\ell}{}^kS^\ell\partial_{x^k}\\
 &=&\{(-R_{121}{}^1-R_{122}{}^2)S^1+(R_{121}{}^1S^1+R_{122}{}^1S^2)\}(dx^1\wedge dx^2)\otimes\partial_{x^1}\\
 &&\mbox{}+\{(-R_{121}{}^1-R_{122}{}^2)S^2+(R_{121}{}^2S^1+R_{122}{}^2S^2)\}(dx^1\wedge dx^2)\otimes\partial_{x^2}\\
 &=&(dx^1 \wedge dx^2)\otimes\{(\rho_{12}S^1+\rho_{22}S^2)\partial_{x^1}-(\rho_{11}S^1+\rho_{21}S^2)\partial_{x^2}\}\,.
\end{eqnarray*}
Consequently
\begin{equation}\label{E1}
\left(\begin{array}{cc}\rho_{11}&\rho_{21}\\
\rho_{12}&\rho_{22}\end{array}\right)\
\left(\begin{array}{cc}S^1\\S^2\end{array}\right)=\left(\begin{array}{c}0\\0\end{array}\right)\,.
\end{equation}
Thus if $\mathfrak{P}$ is non-trivial, $\operatorname{Rank}\{\rho\}\le1$.
Fix $P\in M$ and let $(x^1,x^2)$ be a system of local coordinates centered at $P$. Suppose $\dim\{\mathfrak{P}\}=2$.
Choose $S_i\in\mathfrak{P}(\mathcal{M})$ so $S_1(P)=\partial_{x^1}$ and $S_2(P)=\partial_{x^2}$.
Equation~(\ref{E1}) then implies $\rho=0$.
\end{proof}

\subsection{The associated torsion free surface}
If $\mathcal{M}$ is an affine surface without torsion and if $S\in\mathfrak{T}(M)$,
then we can perturb the Christoffel symbols of $\mathcal{M}$ to create a new affine manifold
$\mathcal{M}(S)=(M,{}^S\nabla)$ with  $S$ as the associated torsion tensor
by setting
\begin{equation}\label{E2}\begin{array}{llll}
{}^S\Gamma_{11}{}^1=\Gamma_{11}{}^1,&{}^S\Gamma_{11}{}^2=\Gamma_{11}{}^2,\\
{}^S\Gamma_{12}{}^1=\Gamma_{12}{}^1+S^1,&{}^S\Gamma_{21}{}^1=\Gamma_{12}{}^1-S^1,&
{}^S\Gamma_{22}{}^1=\Gamma_{22}{}^1,\\
{}^S\Gamma_{12}{}^2=\Gamma_{12}{}^2+S^2,&{}^S\Gamma_{21}{}^2=\Gamma_{12}{}^2-S^2,
&{}^S\Gamma_{22}{}^2=\Gamma_{22}{}^2.
\end{array}\end{equation}
Thus every abstract torsion tensor can be realized geometrically.
Conversely, if $\mathcal{M}$ is an affine manifold with torsion,
set ${}^0\nabla_XY=\nabla_XY-T(X,Y)$ and obtain
an associated surface ${}^0\!\mathcal{M}=(M,{}^0\nabla)$ such that ${}^0\!\mathcal{M}(T)=\mathcal{M}$. We then have
$$
{}^0\Gamma_{ij}{}^k=\textstyle\frac12\{\Gamma_{ij}{}^k+\Gamma_{ji}{}^k\}\,.
$$

Let $\mathcal{M}_{u,v}$ be the geometry with (possibly) non-zero Christoffel symbols
$\Gamma_{11}{}^1=1$, $\Gamma_{12}{}^1=2u$, and $\Gamma_{22}{}^1=v$ for $(u,v)\in\mathbb{R}^2$.
The associated torsion free geometry ${}^0\!\mathcal{M}_{u,v}$ has Christoffel symbols
${}^0\Gamma_{11}{}^1=1$, ${}^0\Gamma_{12}{}^1={}^0\Gamma_{21}{}^1=u$, and ${}^0\Gamma_{22}{}^1=v$.
The torsion tensor of $\mathcal{M}_{u,v}$ is given by
$T=(dx^1\wedge dx^2)\otimes(u\partial_{x^1})$.
We make a direct computation to see
$$\begin{array}{lll}
\rho_{\mathcal{M}_{u,v}}=v\,dx^2\otimes dx^2,&\nabla\rho_{\mathcal{M}_{u,v}}=0,&\nabla T=0,\\[0.05in]
\rho_{\,{}^0\!\mathcal{M}_{u,v}}=(v-u^2)\,dx^2\otimes dx^2,&{}^0\nabla(\rho_{\,{}^0\!\mathcal{M}_{u,v}})=0,&{}^0\nabla T=0.
\end{array}$$
Thus both $\mathcal{M}_{u,v}$ and ${}^0\!\mathcal{M}_{u,v}$ are symmetric affine surfaces; the torsion tensor $T$
of $\mathcal{M}_{u,v}$ is parallel both with respect to $\nabla$ and with respect to ${}^0\nabla$.

\subsection{The proof of Theorem~\ref{T2}~(1)}
Let $\mathcal{M}$ be an affine surface which is flat with parallel non-vanishing torsion.
Fix a point $P$ of $M$. Since $R=0$, we can choose a frame $\{X,Y\}$ for the tangent bundle so that $\nabla X=0$
and $\nabla Y=0$. Let $\{X^*,Y^*\}$ be the corresponding dual frame for the cotangent bundle; we then have dually that
$\nabla X^*=\nabla Y^*=0$. Expand $[X,Y]=aX+bY$. Then
$$
T=-\textstyle\frac12(X^*\wedge Y^*)\otimes[X,Y]
=-\frac12(X^*\wedge Y^*)\otimes(a(x^1,x^2)X+b(x^1,x^2)Y)\,.
$$
Since $X$, $Y$, $X^*$, and $Y^*$ are parallel, the assumption that $\nabla T=0$ implies $a$ and $b$ are constant.
Since $T\ne0$, we can make a linear change of frame to assume $a=0$ and $b=-1$ and hence $[X,Y]=-Y$.
Choose local coordinates $(s,t)$ near $P$ so $Y=\partial_{t}$. Expand
$X=u(s,t)\partial_{s}+v(s,t)\partial_{t}$.
The bracket relation $[X,Y]=-Y$ shows $\partial_tu=0$ and
$\partial_{t}v=1$. Consequently, $X=u(s)\partial_{s}+(v(s)+t)\partial_t$.
Perform a shear and set $\tilde s=s$ and $\tilde t=t+\varepsilon(s)$ where $\varepsilon$ remains to be determined. Then
\begin{eqnarray*}
&&d\tilde s=ds,\quad d\tilde t=dt+\varepsilon^\prime(s)ds,\quad \partial_{\tilde s}=\partial_{s}-\varepsilon^\prime(s)\partial_{t},
\quad\partial_{\tilde t}=\partial_{t},\\
&&X=u(\tilde s)\partial_{\tilde s}
+\{v(\tilde s)+\tilde t-\varepsilon(\tilde s)+u(\tilde s)\varepsilon^\prime(\tilde s)\}\partial_{\tilde t},\qquad\ Y=\partial_{\tilde t}\,.
\end{eqnarray*}
Solve the ODE $v(\tilde s)-\varepsilon(\tilde s)+u(\tilde s)\varepsilon^\prime(\tilde s)=0$ to express
$X=u(\tilde s)\partial_{\tilde s}+\tilde t\partial_{\tilde t}$. Set $x^2=\tilde t$ and choose $x^1=x^1(\tilde s)$ so
$x^1\partial_{x^1}=u(\tilde s)\partial_{\tilde s}$. This expresses $X=x^1\partial_{x^1}+x^2\partial_{x^2}$.
Since $\nabla\partial_{x^2}=0$, we have
$\Gamma_{12}{}^1=\Gamma_{12}{}^2=\Gamma_{22}{}^1=\Gamma_{22}{}^2=0$. We compute:
\begin{eqnarray*}
0&=&\nabla_{\partial_1}(x^1\partial_{x^1}+x^2\partial_{x^2})=(1+x^1\Gamma_{11}{}^1)\partial_{x^1}+x^1\Gamma_{11}{}^2\partial_{x^2},\\
0&=&\nabla_{\partial_2}(x^1\partial_{x^1}+x^2\partial_{x^2})
=x^1\Gamma_{21}{}^1\partial_{x^1}+(x^1\Gamma_{21}{}^2+1)\partial_{x^2}.
\end{eqnarray*}
This defines a Type~$\mathcal{B}$ structure where the only non-zero Christoffel symbols are
$\Gamma_{11}{}^1=\Gamma_{21}{}^2=-(x^1)^{-1}$. On the other hand,
the structure $\mathcal{M}_{1,0}$ has non-zero parallel torsion with vanishing Ricci tensor. Consequently this structure
is isomorphic to $\mathcal{M}_{1,0}$. This establishes Theorem~\ref{T2}~(1).~\qed

\subsection{The proof of Theorem~\ref{T2}~(2)}
Let $\mathcal{M}$ be a symmetric affine surface with parallel non-zero torsion which is not flat. By Lemma~\ref{L7} and Lemma~\ref{L8}, the Ricci tensor $\rho$ of $\mathcal{M}$ is symmetric and has rank 1. Define a
smooth 1-dimensional distribution by setting $\ker(\rho):=\{\xi:\rho(\xi,\eta)=0\ \forall\ \eta\}$.
Suppose $\xi\in\ker(\rho)$. Let $\eta$ be an arbitrary tangent vector field. Since $\nabla\rho=0$, we compute
$$
0=(\nabla\rho)(\xi,\eta)=d\rho(\xi,\eta)-\rho(\nabla\xi,\eta)-\rho(\xi,\nabla\eta)=0-\rho(\nabla\xi,\eta)-0\,.
$$
Consequently, the distribution $\ker(\rho)$ is invariant under $\nabla$. Let $0\ne\xi\in\ker(\rho)$. Choose local coordinates so $\xi=\partial_{x^1}$. We then
have $\rho=\rho_{22}dx^2\otimes dx^2$. Since $\ker(\rho)$ is invariant under $\nabla$,
we may expand
$\nabla_{\partial_{x^1}}\partial_{x^1}=\omega_1\partial_{x^1}$ and $\nabla_{\partial_{x^2}}\partial_{x^1}=\omega_2\partial_{x^1}$.
The commutator of covariant differentiation is given by curvature so
$$(\nabla_{\partial_{x^1}}\nabla_{\partial_{x^2}}-\nabla_{\partial_{x^2}}\nabla_{\partial_{x^1}})\partial_{x^1}
=R_{121}{}^1\partial_{x^1}+R_{121}{}^2\partial_{x^2}
=\rho_{21}\partial_{x^1}-\rho_{11}\partial_{x^2}=0\,.$$
We may also compute directly
\begin{eqnarray*}
&&(\nabla_{\partial_{x^1}}\nabla_{\partial_{x^2}}-\nabla_{\partial_{x^2}}\nabla_{\partial_{x^1}})\partial_{x^1}\
=\nabla_{\partial_{x^1}}\{\omega_2\partial_{x^1}\}-\nabla_{\partial_{x^2}}\{\omega_1\partial_{x^1}\}\\
&=&(\omega_1\omega_2+\partial_{x^1}\omega_2-\omega_2\omega_1-\partial_{x^2}\omega_1)\partial_{x^1}\,.
\end{eqnarray*}
This implies $\partial_{x^1}\omega_2-\partial_{x^2}\omega_1=0$. Consequently, there exists a smooth function
$f$ so that $\omega_1=\partial_{x^1}f$ and $\omega_2=\partial_{x^2}f$. Let $\tilde\xi=e^{-f}\xi$. We then have
$\nabla\tilde\xi=0$ so $\tilde\xi$ is a parallel vector field on $\mathcal{M}$. We replace $\xi$ by $\tilde\xi$
and obtain
$$
\rho=\rho_{22}dx^2\otimes dx^2,\quad \Gamma_{11}{}^1=\Gamma_{11}{}^2=\Gamma_{21}{}^1=\Gamma_{21}{}^2=0\,.
$$
Let $A_{ij}{}^k$ be the Christoffel symbols of ${}^0\!\mathcal{M}$. We adopt the notation of Equation~(\ref{E2}) and obtain
$$\begin{array}{llllll}
A_{11}{}^1=0,&A_{11}{}^2=0,&A_{12}{}^1=T^1,&A_{12}{}^2=T^2,\\
A_{22}{}^1=\Gamma_{22}{}^1,&A_{22}{}^2=\Gamma_{22}{}^2\,.
\end{array}$$
A direct computation shows
\begin{equation}\label{E3}\begin{array}{l}
\rho_{11}=0,\quad\rho_{21}=0,\quad\rho_{12}=2\partial_{x^2}A_{12}{}^2-\partial_{x^1}A_{22}{}^2,\\
\rho_{22}=-2A_{12}{}^2A_{22}{}^1+2A_{12}{}^1A_{22}{}^2-2\partial_{x^2}A_{12}{}^1+\partial_{x^1}A_{22}{}^1\,.
\end{array}\end{equation}
We express the equation $\nabla T=\left(\begin{array}{cc}T_{1;1}&T_{1;2}\\T_{2;1}&T_{2;2}\end{array}\right)$ in terms of the $A$ variables:
$$0=\nabla T=\left(\begin{array}{cc}
\partial_{x^1}(A_{12}{}^1)&A_{12}{}^2A_{22}{}^1-A_{12}{}^1A_{22}{}^2+\partial_{x^2}(A_{12}{}^1)\\
\partial_{x^1}(A_{12}{}^2)&\partial_{x^2}(A_{12}{}^2)\end{array}\right)\,.
$$
Consequently, $\partial_{x^1}(A_{12}{}^1)=0$, $\partial_{x^1}(A_{12}{}^2)=0$, and $\partial_{x^2}(A_{12}{}^2)=0$.
By Lemma~\ref{L7}, $\rho_{12}=\rho_{21}=0$. Since $\partial_{x^2}(A_{12}{}^2)=0$, Equation~(\ref{E3}) implies
$\partial_{x^1}A_{22}{}^2=0$. Thus
\begin{eqnarray*}
&&A_{12}{}^1(x^1,x^2)=a_{12}{}^1(x^2),\quad
A_{12}{}^2(x^1,x^2)=c_{12}{}^2\in\mathbb{R},\\
&&A_{22}{}^2(x^1,x^2)=a_{22}{}^2(x^2)\,.
\end{eqnarray*}
Since $\operatorname{Rank}\{\rho\}=1$, $0\ne\rho_{22}$. We compute
$\rho_{22}+2T_{1;2}=\partial_{x^1}(A_{22}{}^1)$. Since $T_{1;2}=0$, we may conclude $\partial_{x^1}(A_{22}{}^1)\ne0$. We have
$$
0=T_{1;2}=c_{12}{}^2A_{22}{}^1(x^1,x^2)-a_{12}{}^1(x^2)a_{22}{}^2(x^2)+\partial_{x^2}(a_{12}{}^1(x^2))\,.
$$
Since $A_{22}{}^1$ exhibits non-trivial dependence on $x^1$, we have that $c_{12}{}^2=0$. Thus
$$\begin{array}{lll}
A_{12}{}^1(x^1,x^2)=a_{12}{}^1(x^2),&A_{12}{}^2(x^1,x^2)=0,&
A_{22}{}^2(x^1,x^2)=a_{22}{}^2(x^2)\,.
\end{array}$$
We may then compute $0=T_{1;2}=-a_{12}{}^1a_{22}{}^2+(a_{12}{}^1)^\prime$. Let $u=a_{12}{}^1(0)$ and let
$a(x^2)$ be a smooth function so $a(0)=0$ and $a^\prime(x^2)=a_{22}{}^2(x^2)$. We can then solve the ODE
$0=-a_{12}{}^1a_{22}{}^2+(a_{12}{}^1)^\prime$ to see:
$$\begin{array}{lll}
A_{12}{}^1(x^1,x^2)=ue^{a(x^2)},&A_{12}{}^2(x^1,x^2)=0,&A_{22}{}^2(x^1,x^2)=a^\prime(x^2)\,.
\end{array}$$
There are only two non-trivial equations remaining to ensure $\nabla\rho=0$:
\begin{eqnarray*}
0&=&(\partial_{x^1})^2A_{22}{}^1(x^1,x^2),\\
0&=&-2a^\prime(x^2)\partial_{x^1}A_{22}{}^1(x^1,x^2)+(\partial_{x^1}\partial_{x^2})A_{22}{}^1(x^1,x^2)\,.
\end{eqnarray*}
This implies $A_{22}{}^1(x^1,x^2)=b(x^2)+x^1 ve^{2a(x^2)}$ for some constant $v\in\mathbb{R}$. We then compute
$\rho=ve^{2a(x^2)}dx^2\otimes dx^2$. Since $\mathcal{M}$ is not flat, $v\neq 0$. We can renormalize $x^2$
so $\rho_{22}=vdx^2\otimes dx^2$ for $v\neq0$. The non-zero Christoffel symbols are then (renaming $e^{-2a(x^2)}b(x^2)\to b(x^2)$)
$$
\Gamma_{12}{}^1=2u\text{ and }\Gamma_{22}{}^1=b(x^2)+vx^1\,.
$$
We perform a shear and set $y^1=x^1+\alpha(x^2)$ and $y^2=x^2$. We then have $\partial_{y^1}=\partial_{x^1}$
and $\partial_{y^2}=\partial_{x^2}-\alpha^\prime(x^2)\partial_{x^1}$. Consequently
\begin{eqnarray*}
&&\nabla_{\partial_{y^1}}\partial_{y^1}=\nabla_{\partial_{x^1}}\partial_{x^1}=0,\\
&&\nabla_{\partial_{y^2}}\partial_{y^1}=\nabla_{\partial_{x^2}}\partial_{x^1}-\alpha^\prime\nabla_{\partial_{x^1}}\partial_{x^1}=0,\\
&&\nabla_{\partial_{y^1}}\partial_{y^2}=\nabla_{\partial_{x^1}}\partial_{x^2}-\alpha^\prime\nabla_{\partial_{x^1}}\partial_{x^1}
=2u\partial_{x^1}=2u\partial_{y^1},\\
&&\nabla_{\partial_{y^2}}\partial_{y^2}
=\nabla_{(\partial_{x^2}-\alpha^\prime\partial_{x^1})}(\partial_{x^2}-\alpha^\prime\partial_{x^1})\\
&&\quad=(b(x^2)+vx^1)\partial_{x^1}-2u\alpha^\prime\partial_{x^1}-\alpha^{\prime\prime}\partial_{x^1}\,.
\end{eqnarray*}
Choose $\kappa$ so that $(x^1+\kappa)>0$ in a neighborhood of the point in question.
We solve the ODE $b(x^2)-2u\alpha^\prime(x^2)-\alpha^{\prime\prime}(x^2)=v\kappa$ to ensure the only non-zero Christoffel symbols
are $\Gamma_{12}{}^1=2u$ and $\Gamma_{22}{}^1=v(x^1+\kappa)$.
We make the change of variables
$\partial_{z^1}=(x^1+\kappa)\partial_{y^1}$ and $\partial_{z^2}=\partial_{y^2}$. We compute
\begin{eqnarray*}
&&\nabla_{\partial_{z^1}}\partial_{z^1}=(x^1+\kappa)\nabla_{\partial_{y^1}}((x^1+\kappa)\partial_{y^1})=
(x^1+\kappa)\partial_{y^1}=\partial_{z^1},\\
&&\nabla_{\partial_{z^2}}\partial_{z^1}=\nabla_{\partial_{y^2}}((x^1+\kappa)\partial_{y^1})=0,\\
&&\nabla_{\partial_{z^1}}\partial_{z^2}=(x^1+\kappa)\nabla_{\partial_{y^1}}\partial_{y^2}=2u(x^1+\kappa)\partial_{y^1}=2u\partial_{z^1},\\
&&\nabla_{\partial_{z^2}}\partial_{z^2}=\nabla_{\partial_{y^2}}\partial_{y^2}=v(x^1+\kappa)\partial_{y^1}=v\partial_{z^1}\,.
\end{eqnarray*}

The non-zero Christoffel symbols now take the form
$$
\Gamma_{11}{}^1=1,\quad\Gamma_{12}{}^1=2u,\quad\Gamma_{22}{}^1=v\,.
$$
We can rescale $x^2$ to assume $v=\pm1$. We must have $u\ne0$ to ensure the torsion is non-zero.
Replacing $x^2$ by $-x^2$ replaces $u$ by $-u$. We may therefore assume $u>0$ and obtain the structures which are given in Theorem~\ref{T2}~(2).~\qed
\subsection{Distinguishing the structures} The structures $\mathcal{M}_{u,v}$ are all Type~$\mathcal{A}$ structures;
they are invariant under the translation group and are thus homogeneous geometries. The signature of the Ricci tensor
determines the parameter $v$. We suppose $v=\pm1$ as there is only one model in Assertion~(1).
Let ${}^0\!\mathcal{M}_{u,v}$ be the associated torsion free geometry;
${}^0\Gamma_{11}{}^1=1$, ${}^0\Gamma_{12}{}^1={}^0\Gamma_{21}{}^1=u$, and ${}^0\Gamma_{22}{}^1=v$.
We have
$\rho_{\,{}^0\!\mathcal{M}_{u,v}}=v(v-u^2)\rho_{\mathcal{M}_{u,v}}$. Since $v$ is determined by the signature of
$\rho_{\mathcal{M}_{u,v}}$, $u^2$ is an invariant of the affine structure in this context. Since $u>0$, $u$ is determined
and the structures are distinct affine structures.~\qed

\section{The proof of Theorem~\ref{T4}}\label{S3}

Let $\mathcal{M}$ be a Type~$\mathcal{A}$ symmetric surface with non-parallel torsion tensor.
By making a suitable change of basis, we may assume $T=(dx^1\wedge dx^2)\otimes\partial_{x^2}$.
This normalizes the linear changes of coordinates up to the action of the $ax+b$ subgroup of $\operatorname{GL}(2,\mathbb{R})$.
Let $A_{ij}{}^k:=\frac12(\Gamma_{ij}{}^k+\Gamma_{ji}{}^k)$
be the Christoffel symbols of  ${}^0\!\mathcal M$.
The following is a useful result which follows by a direct computation.

\begin{lemma}\label{L9}
Let $(y^1,y^2)=(x^1,a^{-1}(x^2-bx^1))$ be a change of variables which defines a {shear}\index{shear}. Then
\par $dy^1=dx^1$, \quad $dy^2=a^{-1}(dx^2-bdx^1)$,\quad
$\partial_{y^1}=\partial_{x^1}+b \partial_{x^2}$,\quad$\partial_{y^2}=a\partial_{x^2}$,
\par ${}^y\!A_{11}{}^1={}^x\!A_{11}{}^1+2b\,{}^x\!A_{12}{}^1+b^2\,{}^x\!A_{22}{}^1$,
\par ${}^y\!A_{11}{}^2=\frac1a\{{}^x\!A_{11}{}^2+b (2\  {}^x\!A_{12}{}^2-{}^x\!A_{11}{}^1)
+b^2 ({}^x\!A_{22}{}^2-2\ {}^x\!A_{12}{}^1)-b^3\,{}^x\!A_{22}{}^1\}$,
\par $ {}^y\!A_{12}{}^1=a ({}^x\!A_{12}{}^1+b\,{}^x\!A_{22}{}^1)$,
\par ${}^y\!A_{12}{}^2={}^x\!A_{12}{}^2+b\,{}^x\!A_{22}{}^2-b({}^x\!A_{12}{}^1+b\ {}^x\!A_{22}{}^1)$,
\par ${}^y\!A_{22}{}^1=a^2\ {}^x\!A_{22}{}^1$,\par $
{}^y\!A_{22}{}^2=a( {}^x\!A_{22}{}^2 - b\,{}^x\!A_{22}{}^1 )
$.
\end{lemma}

By Lemma~\ref{L9}, if $\Gamma_{22}{}^1\ne0$, we can always fix the gauge so
$\Gamma_{22}{}^1=\pm1$ and $\Gamma_{22}{}^2=0$. If $\Gamma_{22}{}^1=0$, we
can rescale $x^2$ to assume $\Gamma_{22}{}^2\in\{0,1\}$ but the gauge is not yet fixed.  This gives rise to three cases.
We will use a similar gauge normalization in the Type~$\mathcal{B}$ setting.
\smallbreak\noindent{\bf Case 1:} $\Gamma_{22}{}^1=0$ and $\Gamma_{22}{}^2\ne0$.
Rescale $x^2$ to assume $\Gamma_{22}{}^2=1$ and set $a=1$ in Lemma~\ref{L9}.
We compute that $0=\rho_{22;2}=2(A_{12}{}^1-1)A_{12}{}^1$.
There are two subcases:
\smallbreak\noindent{Case 1.1: $A_{12}{}^1=0$.} The only remaining non-zero component of $\nabla\rho$
is given by $\rho_{11;1}=-2A_{11}{}^1[1+A_{11}{}^2+A_{11}{}^1(-1+A_{12}{}^2)-(A_{12}{}^2)^2]$. There are two subpossibilities:
\smallbreak\noindent{Case 1.1.1: $A_{11}{}^1=0$.} We have $\nabla T=0$ so we reject this case.
\smallbreak\noindent{Case 1.1.2: $A_{11}{}^1\neq 0$ and
$1+A_{11}{}^2+A_{11}{}^1(-1+A_{12}{}^2)-(A_{12}{}^2)^2=0$.}
This fixes $A_{11}{}^2$. Choose $b$ in Lemma~\ref{L9} to ensure $A_{12}{}^2=0$; this gives Assertion~(1).
\smallbreak\noindent{Case 1.2: $A_{12}{}^1=1$.} We have $\rho_{12;2}=2(1-A_{12}{}^2)$.
We set $A_{12}{}^2=1$ and obtain $\nabla\rho=0$. Choose $b$ in Lemma~\ref{L9} to ensure $A_{11}{}^1=0$; this gives Assertion~(2).
\smallbreak\noindent{\bf Case 2:} $\Gamma_{22}{}^1=0$ and $\Gamma_{22}{}^2=0$. Since
$\rho_{22;1}=4(A_{12}{}^1)^2$, we have  $A_{12}{}^1=0$.
$\rho_{11;1}=-2A_{11}{}^1(A_{11}{}^1-A_{12}{}^2-1)(A_{12}{}^2-1)$
is the only non-zero component of $\nabla R$. Furthermore, the only non-zero component of $\nabla T$ is
$T_{2;1}=-A_{11}{}^1$ so $A_{11}{}^1\ne0$.
\smallbreak\noindent{Case 2.1: $A_{12}{}^2=A_{11}{}^1-1$.}
If $A_{11}{}^1\neq2$, we set $a=1$ and choose $b$ in Lemma~\ref{L9} so $A_{11}{}^2=0$. Rescaling $x^2$ then
plays no role. This normalizes the gauge and we obtain Assertion~(3). If on the other hand $A_{11}{}^1=2$
and $A_{11}{}^2\ne0$, then we can rescale $x^2$ to obtain Assertion~(4) and again we have fixed the gauge as the parameter $b$
plays no role. Finally, if $A_{11}{}^1=2$ and $A_{11}{}^2=0$, we again obtain Assertion~(3).
\smallbreak\noindent{Case 2.2: $A_{12}{}^2\neq A_{11}{}^1-1$ and $A_{12}{}^2=1$}. Thus $A_{11}{}^1\neq2$ and we can choose the parameter $b$ in Lemma~\ref{L9} so that $A_{11}{}^2=0$. We obtain Assertion~(5).

\smallbreak\noindent{\bf Case 3:} $A_{22}{}^1\ne0$. We use Lemma~\ref{L9} to make a gauge transformation and fix the
gauge so $A_{22}{}^1=\varepsilon=\pm1$ and $A_{12}{}^1=0$. We set $A_{11}{}^1=\omega$ and $A_{22}{}^2=\eta$ and
compute $0=\rho_{12;2}-\rho_{22;1}=-4\varepsilon(1+A_{12}{}^2-\omega)$.
We set $A_{12}{}^2=\omega-1$ and compute $0=\rho_{22;2}=2A_{11}{}^2$. We then have $\nabla\rho=0$ and obtain Assertion~(6).~\qed

\section{The proof of Theorem~\ref {T5}}\label{S4}
The essential technical point in performing the analysis is to fix the gauge; otherwise the problem
is combinatorially intractable. The torsion tensor plays an essential role in this regard.
For Type $\mathcal A$ surfaces we used the action of $\operatorname{GL}(2,\mathbb{R})$ to set $T=(dx^1\wedge dx^2)\otimes\partial_{x^2}$.
The remaining gauge freedom is then governed by the $ax+b$ group sending
$(x^1,x^2)\rightarrow(x^2,ax^2+bx^1)$. The natural gauge group in the Type~$\mathcal{B}$
setting is again the $ax+b$ group with the same action on the coordinates. We denote $\tilde A_{ij}{}^k$ the
Christoffel symbols of ${}^0\!\mathcal M$ evaluated at $x_1=1$; $\tilde A_{ij}{}^k=x^1(\Gamma_{ij}{}^k+\Gamma_{ji}{}^k)/2$.

Let $\mathcal{M}$ be a symmetric affine surface of Type~$\mathcal{B}$. The Ricci tensor
is symmetric. This yields the relation $\tilde A_{12}{}^1=\tilde T^1-\tilde A_{22}{}^2$.
This is analogous to using the general linear group in the Type~$\mathcal{A}$ setting
to fix the gauge. The $ax+b$ group now acts and we have the same 3 cases as in Lemma~\ref{L9} in
the Type~$\mathcal{A}$ setting. Note that $\tilde T^2$ is still a free parameter.
\smallbreak\noindent{\bf Case 1:} $\tilde A_{22}{}^1=0$ and $\tilde A_{22}{}^2\ne0$.
We rescale $x^2$ to assume $\tilde A_{22}{}^2=1$. We then have $\tilde\rho_{22;2}=-4(2\tilde T^1-1)$ so $\tilde T^1=\frac12$ and
 $\tilde A_{12}{}^1=\tilde T^1-\tilde A_{22}{}^2=-\frac12$. Since
$\tilde\rho_{12;1}=-2-\tilde A_{11}{}^1-\tilde A_{12}{}^2-\tilde T^2$, we obtain $\tilde A_{11}{}^1=-2-\tilde A_{12}{}^2-\tilde T^2$. We finally compute
$\tilde\rho_{11;2}=4(\tilde A_{11}{}^2-\tilde A_{12}{}^2-(\tilde A_{12}{}^2)^2+\tilde T^2+(\tilde T^2)^2)$, which leads to
$$
\tilde A_{11}{}^2=\tilde A_{12}{}^2+(\tilde A_{12}{}^2)^2-\tilde T^2-(\tilde T^2)^2\,.
$$
We now have $\nabla\rho=0$. 
Since $\tilde A_{22}{}^2\neq 0$, we can make a shear to set $\tilde A_{12}{}^2+\tilde T^2=0$. We thus obtain Assertion~(1).

\smallbreak\noindent{\bf Case 2:}  $\tilde A_{22}{}^1=0$ and $\tilde A_{22}{}^2=0$.
We have $\tilde\nabla\rho_{22;1}=-8(\tilde T^1)^2(\tilde A_{12}{}^2-\tilde T^2)$. This gives rise to 2 cases.
\smallbreak\noindent{Case 2.1: $\tilde A_{12}{}^2\neq \tilde T^2$}. Thus $\tilde T^1=0$ so
$\tilde A_{12}{}^1=0$.
The only remaining non-zero component of $\nabla\rho$ is given by
$\tilde\rho_{11;1}=2(1+\tilde A_{11}{}^1)(\tilde A_{12}{}^2-\tilde T^2)(-1-\tilde A_{11}{}^1+\tilde A_{12}{}^2+\tilde T^2)$.
If $\tilde A_{11}{}^1=-1$ we have $\nabla T=0$. We take $\tilde A_{12}{}^2=1+\tilde A_{11}{}^1-\tilde T^2$. This ensures $\nabla\rho=0$.
We now fix the gauge.
\smallbreak\noindent{Case 2.1.1: $\tilde A_{11}{}^2=0$.} We obtain Assertion~(2).
\smallbreak\noindent{Case 2.1.2: $\tilde A_{11}{}^1\ne2\tilde T^2-2$.} We have
$\tilde A_{22}{}^2=\tilde A_{12}{}^1=\tilde A_{22}{}^1=0$. Furthermore $2\tilde A_{12}{}^2-\tilde A_{11}{}^1=2+\tilde A_{11}{}^1-2\tilde T^2\ne0$. Thus we can use
Lemma~\ref{L9} to make a gauge transform to ensure $\tilde A_{11}{}^2=0$ which reduces to Case 2.1.1.
\smallbreak\noindent{Case 2.1.3: $\tilde A_{11}{}^2\ne0$ and $\tilde A_{11}{}^1=2\tilde T^2-2$.}
Rescale $x^2$ to ensure $\tilde A_{11}{}^2=1$. The shear parameter $b$ in Lemma~\ref{L9} plays no role. We obtain Assertion~(3).

\smallbreak\noindent{Case 2.2: $\tilde A_{12}{}^2=\tilde T^2$}.
We then have $\nabla\rho=0$. We fix the gauge.
\smallbreak\noindent{Case 2.2.1: $\tilde T^1\ne0$.} We have $\tilde A_{12}{}^1=\tilde T^1-\tilde A_{22}{}^2=\tilde T^1\ne0$.
Since $\tilde A_{22}{}^1=0$, we can choose $b$ in Lemma~\ref{L9} to assume $\tilde A_{11}{}^1=0$. We can then rescale $x^2$
to assume $A_{12}{}^1=\frac12$ and obtain Assertion~(4).
\smallbreak\noindent{Case 2.2.2: $\tilde T^1=0$ and $\tilde A_{11}{}^2=0.$} We obtain
Assertion~(5); the remaining gauge freedom plays no role.
\smallbreak\noindent{Case 2.2.3: $\tilde T^1=0$ and $\tilde A_{11}{}^1\neq 2\tilde T^2$.}
We have $\tilde A_{22}{}^1=0$, $\tilde A_{12}{}^1=0$, $\tilde A_{22}{}^2=0$, and $2\tilde A_{12}{}^2-\tilde A_{11}{}^1\ne0$. We can therefore
apply Lemma~\ref{L9} to choose $b$ so $\tilde A_{11}{}^2=0$ and obtain Case 2.2.2.
\smallbreak\noindent{Case 2.2.4: $\tilde T^1=0$, $\tilde A_{11}{}^1=2\tilde T^2$ and
$\tilde A_{11}{}^2\ne0$.} We rescale $x^2$ to obtain Assertion~(6).

\smallbreak\noindent{\bf Case 3:} $\tilde A_{22}{}^1\ne0$.
We may rescale $x^2$ and then use Lemma~\ref{L9} to assume $\tilde A_{22}{}^1=\varepsilon$ and $\tilde A_{22}{}^2=0$ for
$\varepsilon=\pm1$.
We have
$0=\tilde\rho_{12;2}=2\varepsilon(\tilde A_{12}{}^2-\tilde T^2)(\tilde T^2+\tilde A_{12}{}^2-\tilde A_{11}{}^1)$.
\smallbreak\noindent{Case 3.1: $\tilde A_{12}{}^2=\tilde T^2$.} We compute $\tilde\rho_{22;2}=2\tilde A_{11}{}^2$.
We set $\tilde A_{11}{}^2=0$; the only remaining equation is
$\tilde\rho_{22;1}=-2\varepsilon(-1+\tilde A_{11}{}^1-2\tilde T^2)(1+2\tilde T^2)$.
\smallbreak\noindent{Case 3.1.1: $\tilde T^2=-\frac12$.} We obtain Assertion~(7)
\smallbreak\noindent{Case 3.1.2: $\tilde T^2\neq-\frac12$ and $\tilde A_{11}{}^1=1+2\tilde T^2$.} We obtain Assertion~(8).

\smallbreak\noindent{Case 3.2: $\tilde A_{12}{}^2\neq\tilde T^2$ and
$\tilde A_{12}{}^2=\tilde A_{11}{}^1-\tilde T^2$.} We obtain
$$
\tilde\rho_{22;2}=2\varepsilon(-2\tilde A_{11}{}^1\tilde T^1+4\tilde T^1\tilde T^2+\varepsilon\tilde A_{11}{}^2)\,.
$$
This determines $\tilde A_{11}{}^2$.
We have $\tilde\rho_{22;1}=2\varepsilon(1+\tilde A_{11}{}^1)$ and hence $\tilde A_{11}{}^1=-1$.
To ensure that $\tilde A_{12}{}^2\ne \tilde T^2$, we require that $\tilde T^2\ne-\frac12$. We obtain Assertion~(9).~\qed


\subsection*{Acknowledgments}
Research of DD was partially supported by Universidad Nacional de La Plata under grant 874/18 and project 11/X791. Research of PBG was partially supported by Project MTM2016-75897-P (AEI/FEDER, Spain). Research of PP was partially supported by a Fulbright-CONICET scholarship and by Universidad Nacional de La Plata under project 11/X615. DD and PP thank the warm hospitality at the Mathematics Department of the University of Oregon, where part of this work was carried out.


\bibliography{DGP-arXiv}{}
\bibliographystyle{plain}


\end{document}